\documentclass[12pt]{article}
\usepackage{amsmath,amssymb,amsfonts,amsthm,url,xspace,graphicx,color}
\newtheorem{theorem}{Theorem}

\numberwithin{theorem}{section}

\newtheorem{lemma}[theorem]{Lemma}

\newtheorem{rmk}{Remark}

\newcommand{\A}{{\mathbb A}}

\newcommand{\Z}{\mathbb{Z}}
\newcommand{\R}{\mathbb{R}}

\newcommand{\G}{\mathcal{G}}

\renewcommand{\H}{\mathbb{H}}

\renewcommand{\Im}{\mathrm{Im}}
\newcommand{\eps}{\varepsilon}
\newcommand{\Arg}{\mathrm{Arg}}
\newcommand{\be}{\begin{equation}}
\newcommand{\ee}{\end{equation}}
\newcommand{\old}[1]{}

\newcommand\blfootnote[1]{%
  \begingroup
  \renewcommand\thefootnote{}\footnote{#1}%
  \addtocounter{footnote}{-1}%
  \endgroup
}

\begin{document}
\title{Limit shapes from harmonicity: dominos and the five vertex model}
\author{Richard Kenyon\thanks{Department of Mathematics, Yale University, New Haven CT 06520; richard.kenyon at yale.edu. Research supported by NSF DMS-1940932 and the Simons Foundation grant 327929.} \and Istv\'an Prause\thanks{
Department of Physics and Mathematics, University of Eastern Finland, P.O. Box 111, 80101 Joensuu, Finland; istvan.prause at uef.fi. Research supported by Academy of Finland grant 355839 and V\"ais\"al\"a foundation travel grant.}}

\date{}
\maketitle
\abstract{
We discuss how to construct limit shapes for the domino tiling model (square lattice dimer model)
and $5$-vertex model, in appropriate polygonal domains. Our methods are based on the harmonic
extension method of [R. Kenyon and I. Prause, Gradient variational problems in $\R^2$, Duke Math J. 2022].}


\section{Introduction}

For a convex function $\sigma:\R^2\to\R\cup\{\infty\}$, called the \emph{surface tension}, 
the associated \emph{gradient variational problem} is the problem of minimizing the surface tension integral for functions $h:U\to\R$ on a domain $U\subset\R^2$
\be\label{st}\min_h\iint_U \sigma(\nabla h(x,y))\,dx\,dy,~~~~~~h|_{\partial U} = h_0.\ee
In other words this is the problem of finding
a function $h:U\to\R$
with prescribed boundary values $h|_{\partial U} = h_0$ and minimizing the surface tension integral. 

Well-known examples of gradient variational problems are the Dirichlet problem for the Laplacian, where $\sigma(\nabla f) = |\nabla f|^2$,
and the minimal surface equation, where
$\sigma(\nabla f) = \sqrt{1+|\nabla f|^2}.$

In statistical mechanics, many \emph{limit shape problems} are gradient variational problems, for example
the limit shapes for dimers \cite{CKP} and the five- and six-vertex models \cite{dGKW, PR}. 

In \cite{KP1} we showed that any gradient variational problem (\ref{st}) can be reduced to a harmonic extension problem:
any surface tension minimizer is the envelope of a family of planes in $\R^3$ whose
defining coefficients are \emph{$\kappa$-harmonic functions} of the intrinsic (isothermal) coordinate $z$. 
Here $\kappa$ is the square root of the determinant of the Hessian $\text{Hess}(\sigma)$ with respect to the intrinsic coordinate;
by \emph{$\kappa$-harmonic function} we mean a function $f: \Omega \to\R$, $\Omega \subset \R^2$ satisfying
\be\label{kappa}\nabla\cdot \kappa\nabla f = 0.\ee
This is called the inhomogeneous-conductance Laplace equation \cite{AIM}, and is the analog of a harmonic 
function in the presence of a varying positive conductance field $\kappa=\kappa(x,y)$. 

In the current paper we apply this procedure to get explicit limit shapes for two particular statistical mechanics models:
the square grid dimer model (domino tiling model) and $5$-vertex model. 

Limit shapes for the dimer model were originally discussed by Cohn, Kenyon and Propp in \cite{CKP} and Kenyon and Okounkov in \cite{KO1}. See also more recent work by Astala, Duse, Prause, and Zhong in \cite{ADPZ}, where general results are obtained about geometry and regularity of limit shapes. 
An explicit construction for limit shapes 
was given for a particular dimer model, the honeycomb dimer model, in \cite{KO1}. This is,
in a well-defined sense, the simplest dimer model. 
Extending the ideas of \cite{KO1} to other dimer models such as the square grid dimer model, however, 
has been problematic until recently, primarily 
due to two difficulties: higher degree spectral curves and the problem of ``matching parameters". The work \cite{ADPZ} does apply to higher degree spectral curves and indeed for quite general dimer models but it is of limited value in finding a limit shape matching a given set of boundary conditions.
Using the recently developed method of \cite{KP1} of harmonic extension we
can now carry out this matching procedure for a larger class of dimer models:
the so-called isoradial, or genus-zero cases (which include the domino cases). 
Note that the recent work \cite{BB} solves the limit shape problem for more general higher genus dimer
models but only for a specific domain, the ``Aztec diamond''. For comparison, we illustrate our method on a genus one example in Section \ref{se:aztecfortress}.

The dimer model has the special feature that $\kappa$ of equation (\ref{kappa}) is constant,
so $\kappa$-harmonic functions are simply harmonic functions. Moreover for the standard polygonal domains (defined below),
we have \emph{a priori} knowledge about the tangent planes to the limit shapes along $\partial U$, so the corresponding harmonic
extension problem can be quite explicitly carried out. This is explained in Section \ref{dominoscn} below, and an explicit worked example is given.

The $5$-vertex model is a generalization of the dimer model for which $\kappa$ is not constant;
nonetheless it has a special property, called the ``trivial potential" property \cite{KP1}, that $\sqrt{\kappa}$ is a harmonic function
of the intrinsic coordinate $z$. This implies \cite{KP1} that a $\kappa$-harmonic function
is the ratio of a harmonic function and a certain fixed function $\theta=\theta(z)$. This fact allows us to likewise give an essentially complete
limit shape theory for the five-vertex model, again for the appropriate polygonal domains.

We show in the last section, Section \ref{4vtx}, that the 4-vertex model of \cite{burenev2023arctic} also has trivial potential, and in fact its
limit shapes are obtained from the lozenge dimer model limit shapes by a linear transformation.

We should mention here that limit shapes for most of the basic statistical mechanical models are not known, 
even for the six-vertex model. To find limit shapes one needs an explicit expression for the surface tension, 
or free energy, for general parameters, and this is not usually available, especially for ``non-free fermionic'' models. 
The method presented in this paper applies in principle also for the six-vertex model, replacing harmonic functions 
with $\kappa$-harmonic functions. The key difficulty is that $\kappa$ is not known in this case. 
Finally, we should mention the effect of $\kappa$ on fluctuations: we expect these to be described 
by an \emph{inhomogeneous} Gaussian free field, see \cite{BD}, associated to the $\kappa$-Laplacian 
operator of \eqref{kappa} in the intrinsic complex structure given by the $z$-coordinate.

\old{
We show furthermore that 
there is a surprising relation between domino tiling limit shapes and $5$-vertex model limit shapes:
there is a certain ``duality-and-translation'' operation transforming the arctic curve of the domino limit shape to that of the $5$-vertex limit shape on an associated domain.  
}

\section{Domino limit shapes}\label{dominoscn}

See \cite{Kenyon.lectures} for background on domino tilings and dimers in general. 
We discuss here the square grid dimer model, or domino tiling model. This material can be easily extended to the case of general periodic isoradial dimers (defined in \cite{K.isoradial}), however for simplicity of notation we deal with only the domino case here.

The features of the domino model we need for the current discussion are the following.
\begin{enumerate} 
\item Let $N = [0,1]^2$, the ``Newton polygon" (see \cite{KOS}). The surface tension $\sigma=\sigma(s,t)$ of (\ref{st}) 
is analytic on $N$, $+\infty$ outside of $N$, and strictly convex on the interior of $N$ (the so-called \emph{rough phase}, or \emph{liquid phase}). Thus our minimizers $h$
have gradient $\nabla h=(s,t)$ constrained to lie in $N$.
\item The Hessian of $\sigma$ defines on $N$ a Riemannian metric \cite{KP1}. The intrinsic coordinate $z$ (the conformal coordinate) for this metric can and will be chosen to be parameterized by the upper half plane: $z\in\H$. There is another intrinsic coordinate $w$, where $w\in\bar\H$, the lower half plane, related to $z$ through the ``spectral curve" $P(z,w) = 1+z+w-zw=0$.
See Figure \ref{dominozw}.
\item For $(s,t)\in N$ the relation between $s,t$ and $z$ is 
\be\label{stdef}(s,t) = \frac{1}{\pi}(\pi-\arg z,\pi+\arg w).\ee
See Figure \ref{dominozw}.
\end{enumerate}

\begin{figure}
\begin{center}\includegraphics[width=2in]{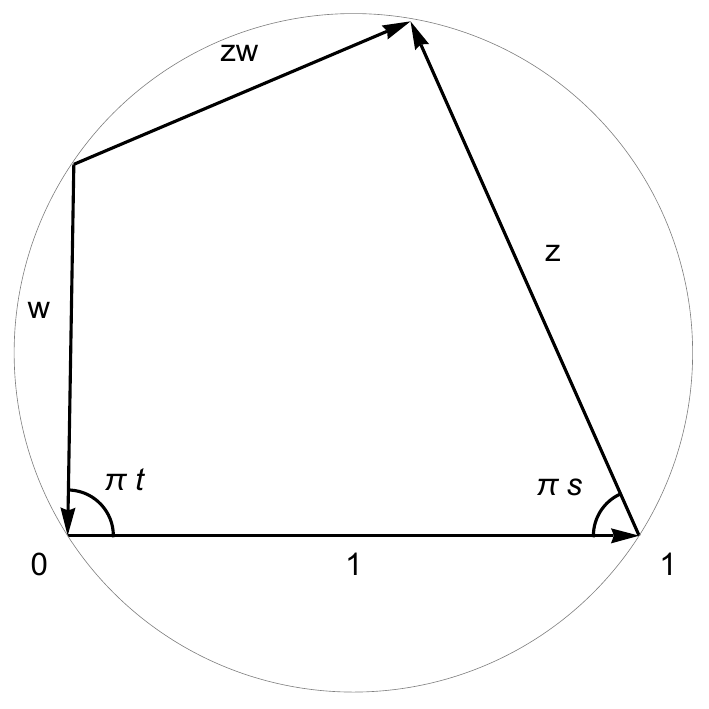}\end{center}
\caption{\label{dominozw}Relation between $z,w,s,t$ for dominos. Given $z\in\H$, draw a circle passing through $0,1$ and $1+z$. Then $w,s,t$ are determined as shown.}
\end{figure} 

We orient edges of $N$ counterclockwise, and label them cyclically $1,2,3,4$ with $1$ being the lower edge.
Let $R_0\subset\R^2$ be a polygon with $4n$ sides, with edges orthogonal to those of $N$, and occurring in 
\emph{clockwise} order around the boundary. That is,
edges of $R_0$ are oriented clockwise around $R_0$ and 
are labelled $1,2,3,4$ cyclically: $1,2,3,4,1,2,3,4,\dots$, with edges
of label $i$ associated to, and orthogonal to, edge $i$ of $N$. See Figure \ref{dominorgn}
for an example. Generally we can allow some edges of $R_0$ to have length $0$. 

Let $\G=e^{i\pi/4}\Z^2$, the square grid rotated by $\pi/4$.
We consider dimer covers of regions $R=R_\eps$ in $\eps\G$, approximating $R_0$ as $\eps\to0$, such that
edges of $R$ consist of zig-zag paths of $\eps\G$
as in Figure \ref{dominorgn}. For example a ``vertical boundary" of $R$ of type $3$ corresponds to an alternating sequence
of NW- and NE-edges of $\G$. 
\begin{figure}
\begin{center}\includegraphics[width=2.2in]{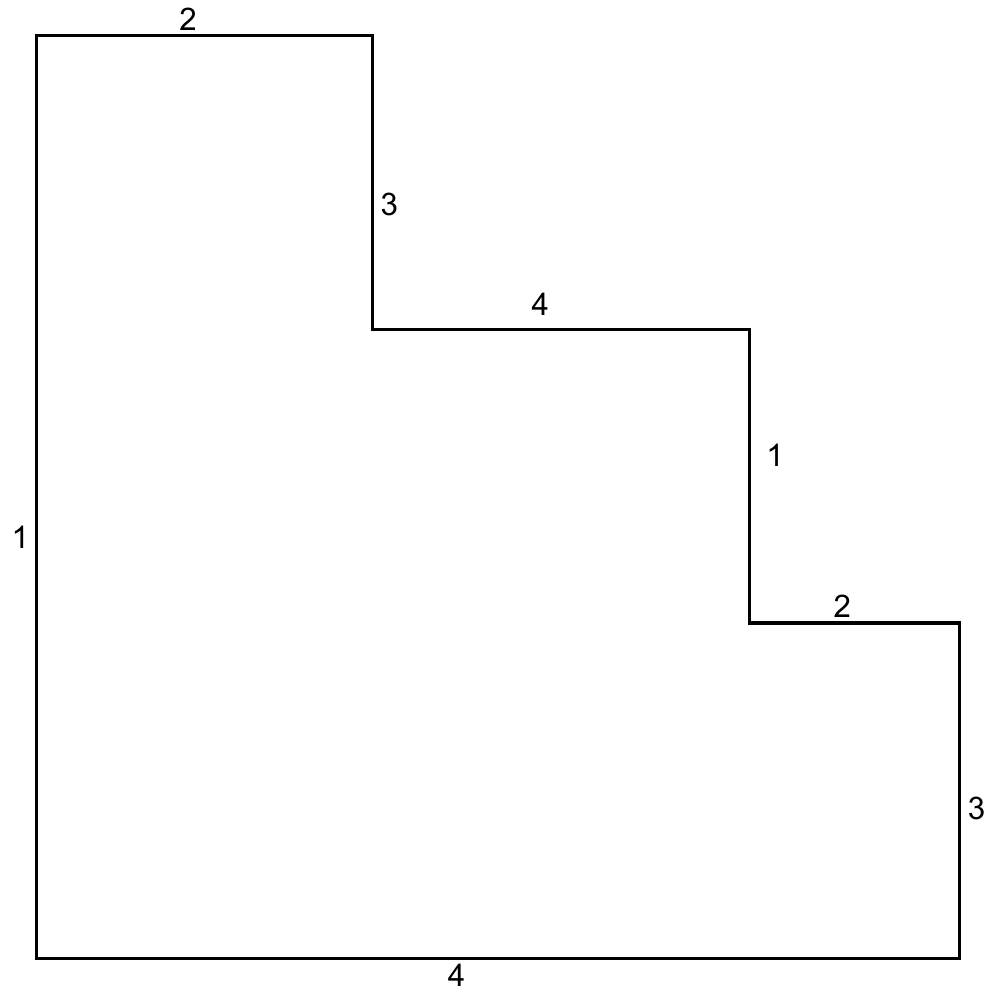}\hskip1cm\includegraphics[width=2.5in]{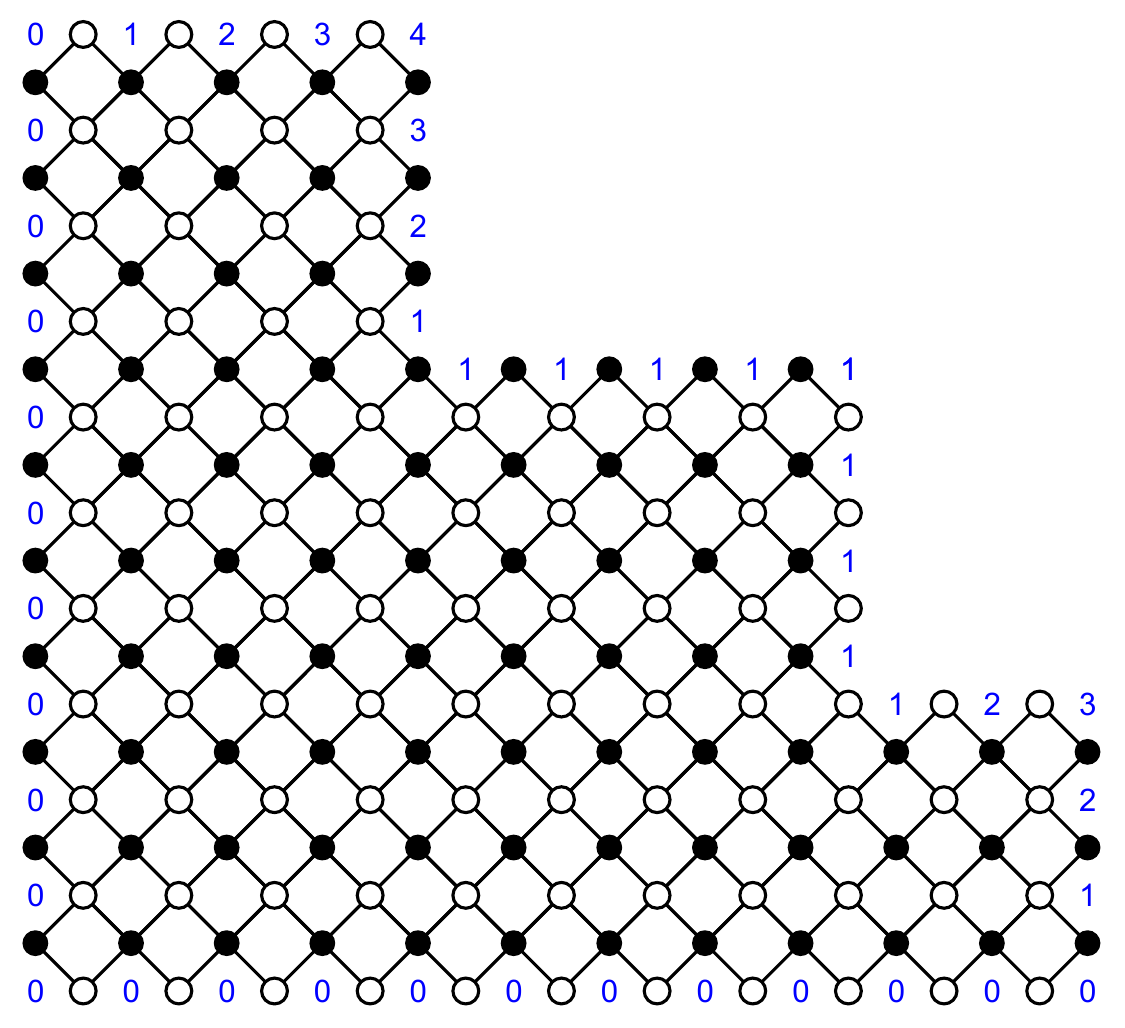}\end{center}
\caption{\label{dominorgn}Polygon $R_0$ and approximating graph $\G$ with boundary height function (in blue).}
\end{figure} 
An edge of $R$ of type $1$ is vertical and either oriented upwards or downwards (when moving clockwise around $R$); 
if upwards it has black vertices on the outside and if downwards it has white vertices on the outer side of $R$, as in Figure \ref{dominorgn}. 
Similar conditions hold for the other types of edges.
To each corner of $R$ is associated a vertex of $N$. 

The height function\footnote{This version of the height function 
is defined by comparing a dimer cover $m$ of $R$ with the fixed dimer cover $m_0$ of $\G$ which uses all dimers
on $\G$ which are oriented NW/SE with black NW vertex. 
The union $m\cup m_0$ consists in doubled edges and oriented chains of dimers, which are oriented so that dimers from $m_0$ are oriented from black to white. 
The height, defined on the faces of $R$, is constant on complementary components of the chains
and increases by one when crossing a chain oriented from left to right. See \cite{Kenyon.lectures}.} on the boundary of $R$ is defined on faces along the boundary, is constant on edges of types $1$ and $4$, and increasing rightwards on edges of
type $2$ and upwards on edges of type $3$, again as in Figure \ref{dominorgn}.
The fact that the height function is well-defined around the boundary is a necessary condition for existence of a dimer cover
(it implies that the graph is balanced, that is, has the same number of white vertices as black vertices). The  
condition of being balanced is a linear condition on the signed sidelengths: the total signed length of edges of type $2$ equals the total signed length of edges of type $3$ (where we define the length of these edges to be the number of vertices along them, not including vertices at concave corners).
 
Suppose $R_0$ is as above,
and let $R$ be the associated approximating region in $\eps\G$. 
We translate $R_0$ and $R$ so that $(0,0)$ is at a convex corner of $R_0$, with sides $1,2,\dots,4n$ in clockwise order starting from $(0,0)$.
We orient $R$ and $R_0$ so that side $1$ is vertical upwards, of type $1$, and side $4n$ is horizontal leftwards (of type $4$). 
The slopes of the height function of a dimer cover of $R$ near the vertices of $R$, ordered clockwise, 
run consecutively counterclockwise around the vertices of $N$, and are periodic with period $4$:
$$(0,0),(1,0),(1,1),(0,1),(0,0),\dots.$$

Letting $L_i$ be the signed length of side $i$ of $R_0$, the tangent planes to the limit shape in $R_0$ near the corners of $R_0$ are determined by these slopes and heights:
$$0, x, x+y-L_1,x+ L_2-L_1, L_2-L_3, \dots.$$

The rough phase region in the limit shape (the region where the gradient lies in the interior of $N$) is parameterized by $u\in\H$; the diffeomorphism from the rough phase region
to $u\in\H$ \emph{reverses orientation} (this is just conventional; we could instead
parameterize the rough phase region with the lower half-plane, and then orientation would be preserved).
Here $u$ is an $n$-fold cover of $z\in\H$, that is, $z=z(u)$ is a rational function of degree $n$. We 
let $a_1, a_2,\dots,a_{4n}\in\hat\R$ be the points $u$, in reverse order around $\hat\R$, at which the rough phase region 
is tangent to the (extended) lines corresponding to the sides of $R_0$.
These are points at which $z$ takes value $\{-1,0,1,\infty,-1,0,1,\infty,\dots\}$.
Thus $z$ necessarily has the form
$$z=B_1\frac{(u-a_2)(u-a_6)\dots(u-a_{4n-2})}{(u-a_4)(u-a_8)\dots(u-a_{4n})}$$
for a constant $B_1$.
Likewise we have
$$w=\frac{B_2(u-a_1)(u-a_5)\dots(u-a_{4n-3})}{(u-a_3)(u-a_7)\dots(u-a_{4n-1})}.$$
From $P(z,w)=0$ we have $w=\frac{z+1}{z-1}$ which can be used to determine $B_2, a_1,a_3,\dots,a_{4n-1}$ 
as functions of $B_1,a_2,a_4,\dots,a_{4n}$.

The equation 
\be\label{sutucu}s_ux+t_uy+c_u=0\ee
parameterizes the tangent lines to the arctic curve as $u$ runs over $\hat\R$, see \cite{KP2}.
The roots of $s_u$ therefore correspond to the horizontal sides of $R$, the roots of $t_u$ correspond to the vertical lines of $R$,
and the roots of $c_u$ are the tangent lines through the origin. 

There is one extra necessary condition on the $a_i$, which is that we need $c_u$ to vanish at the critical points of $z=z(u)$:
\begin{lemma}\label{critpoint}
The quantity $c_u$ vanishes at the critical points of $z(u)$. At higher order critical points $c_u$ vanishes with the corresponding higher multiplicity.
\end{lemma}

\begin{proof} The field of tangent planes to the limit shape is the family of planes $\{x_3-sx-ty-c=0~|~u\in\H\}$ where $x_3$
denotes the third coordinate, and $s=s(z), t=t(z), c=c(u)$ with $z=z(u)$ rational as above. The limit shape surface is defined
as the envelope of these planes, that is, as the solution to the simultaneous system
\begin{align*}
x_3&=sx+ty+c\\
0&=s_u x+t_u y+c_u.
\end{align*}
The second equation here is complex and thus gives two real (and linear) relations on $x,y$. 
At a critical point of $z$, $s_u=\frac{ds}{dz} \frac{dz}{du}= 0$ and likewise $t_u=\frac{dt}{dz} \frac{dz}{du}= 0$. So there is a finite value of $(x,y)$ only if $c_u=0$ at this point as well. Likewise for higher multiplicity.
\end{proof}


What remains is to find the parameters $a_i$ so that $c$ satisfies the condition of Lemma \ref{critpoint} and the values of $s_u,t_u,c_u$ at $a_1,a_2,\dots,a_{4n}$ in (\ref{sutucu})
match the lines of the corresponding sides of $R_0$. To carry this out in practice it seems unavoidable to have to solve an algebraic system of equations; however existence and uniqueness of the limit shape imply that
there is a unique solution for which the $a_i$ are cyclically ordered clockwise on $\hat\R$.
We illustrate the method with the two easiest cases of $n=1$ and $n=2$ in the next sections. 

\subsection{Aztec diamond}

The simplest case is when $n=1$ and thus $R_0$ is a square; this is the case of the Aztec diamond, 
first defined in \cite{EKLP1}. See Figure \ref{AD2d}. In this case we can take $z=u$, since the cover is of degree $1$. 
The $s,t,c$ functions as functions of $z$ for $z\in\R$ are (in blue in the figure): 
\begin{center}\begin{tabular}{c | c | c | c | c}
&$-1<z<0$&$0<z<1$&$1<z<\infty$&$-\infty<z<-1$\\\hline
&$0>w>-1$&$-1>w>-\infty$&$\infty>w>1$&$1>w>0$\\\hline
$s$&$0$&$1$&$1$&$0$\\
$t$&$0$&$0$&$1$&$1$\\
$c$&$0$&$0$&$-1$&$0$
\end{tabular}\end{center}
\begin{figure}
\begin{center}\includegraphics[width=2in]{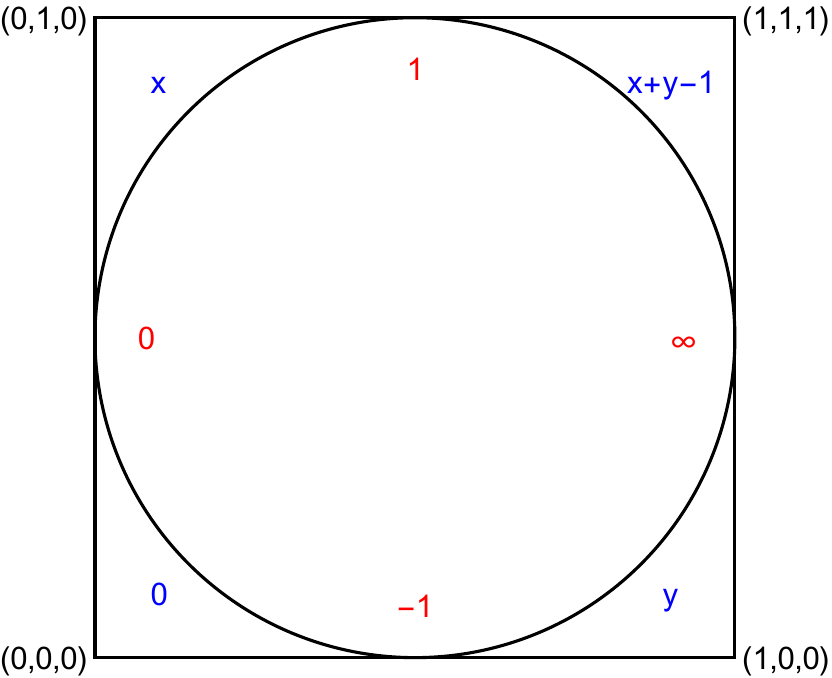}\end{center}
\caption{\label{AD2d}Aztec diamond with $z$ values in red and boundary tangent planes in blue (the blue quantities are the linear equations for the facets: for example the upper right facet lies in the plane $x_3=x+y-1$). The quantities $s,t,c$ for each facet
are the coefficient of $x$, the coefficient of $y$, and the constant term, respectively.  Vertices of the outer polygon are labeled
with their coordinates in $\R^3$ (the third coordinate being the height function).} 
\end{figure} 

Since $c$ is harmonic in $z\in\H$ it is defined from its boundary values 
$$c = \frac{-\pi+\Arg(z-1)}{\pi}.$$
The $s,t$ values are also harmonic; notice that they satisfy (\ref{stdef}).

In this case there is no condition from Lemma \ref{critpoint}. 
The surface can be found as the envelope of the planes
$x_3=sx+ty+c$ as $z$ varies over $\H$. We can find $x,y,h=x_3$ as a function of $z$ by
solving the simultaneous linear equations (with coefficients which are functions of $z$) for $h,x,y$:
\begin{align}
sx+ty+c&=h\notag\\
s_zx+t_zy+c_z&=0\label{3deq}\\
s_{\bar z}x+t_{\bar z}y+c_{\bar z}&=0\notag
\end{align}
Here the third equation is redundant as it follows from the second one, however it is useful to record for computational purposes.
Solving equations (\ref{3deq}) we find
$$(x,y)=\left(\frac{|z|^2}{1+|z|^2},\frac{|z+1|^2}{2(1+|z|^2)}\right),$$
and 
$$h=-1+\frac{\Arg(z-1)}\pi + \frac{|z|^2(\pi-\Arg(z))}{\pi(1+|z|^2)}+\frac{|z+1|^2(\pi+\Arg w)}{2\pi(1+|z|^2)}.$$
See Figure \ref{AD3d}. This Aztec diamond example was treated with the same method in Section 6.1 of \cite{KP1}, but with a different convention for the Newton polygon $N$.
\begin{figure}
\begin{center}\includegraphics[width=3in]{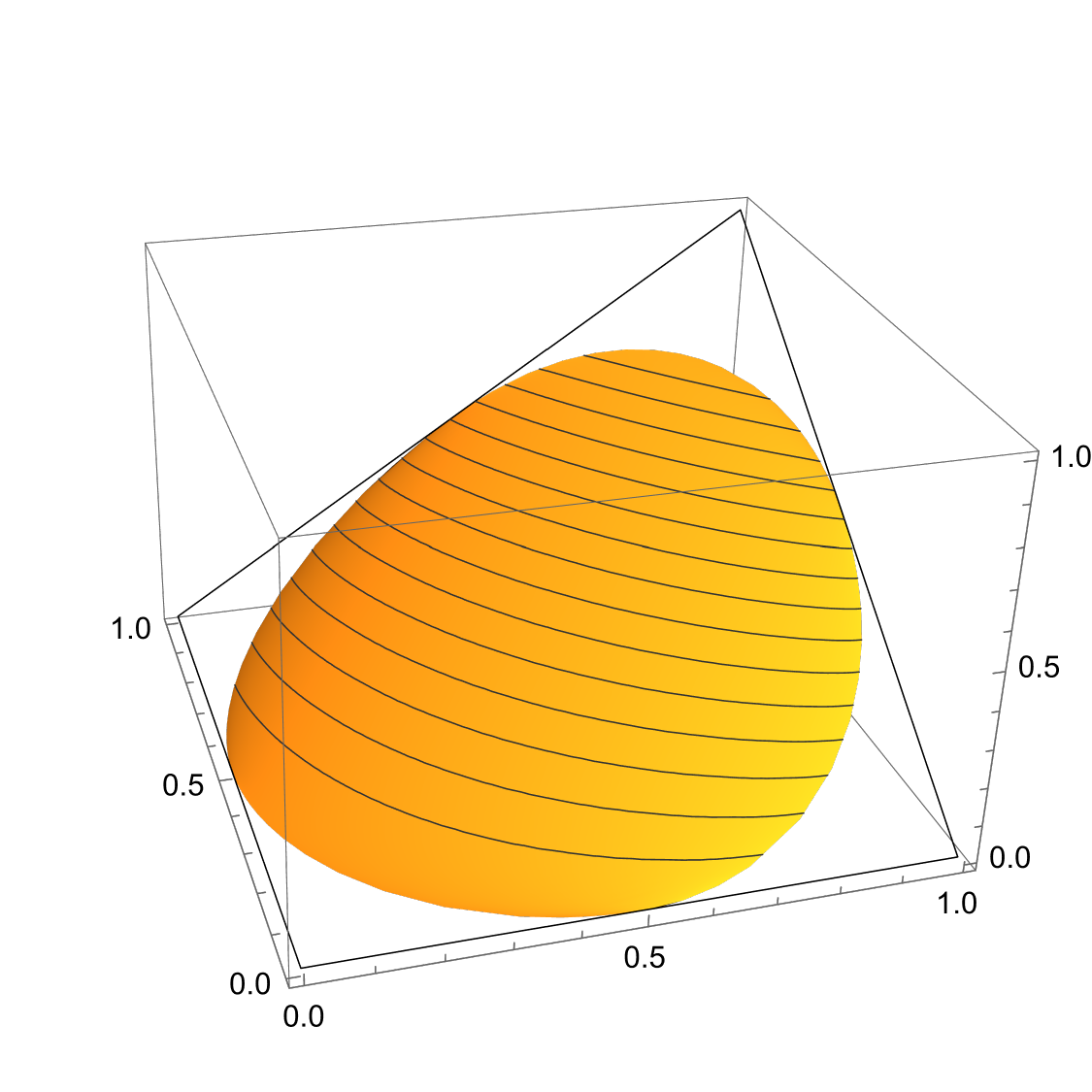}\end{center}
\caption{\label{AD3d} Aztec diamond limit shape surface with contour lines for the height function (the vertical axis is the height function). (Shown is the rough phase 
region; the limit shape is 
linear on each component outside the rough phase region.)}
\end{figure} 

\subsection{Symmetric octagon}

Suppose $R_0$ is the octagon of Figure \ref{dominorgn2} which has a reflection symmetry along the line $x=y$.
(This symmetry is not essential for the theory but makes the calculations easier.)
\begin{figure}
\begin{center}\includegraphics[width=3in]{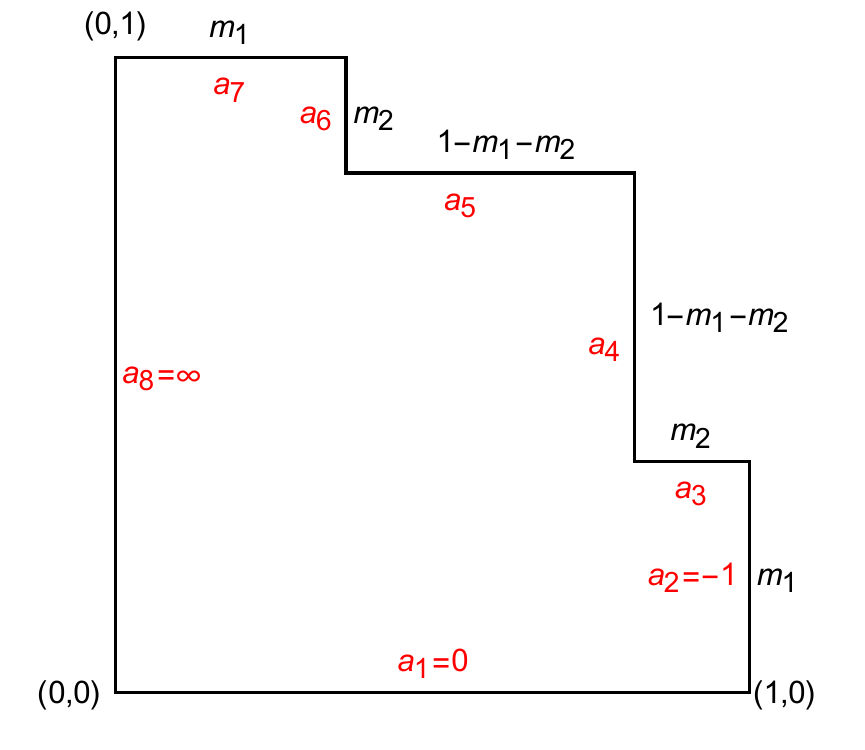}\end{center}
\caption{\label{dominorgn2} An octagon with diagonal symmetry, with $u$ values at tangency points marked in red.}
\end{figure} 

We have $8$ parameters $a_1,\dots,a_8$.
We have the freedom of a M\"obius transformation to set $a_1=0,a_2=-1,a_8=\infty$.
By symmetry of $R_0$, we then necessarily have $a_i=\rho/a_{9-i}$ for some $\rho$ to be determined as a function of the edge lengths $m_1,m_2$.
Set $z=B_1\frac{(u-a_4)}{(u+1)(u-a_6)}$.
Then the fact that $z=-1$ at $u=0$ determines $B_1$ to be $B_1=a_6/a_4$. Likewise $z=-1$ at $u=a_5=\rho/a_4$ determines $\rho$ as a function of $a_4,a_6$:
$\rho=a_4a_6-a_4+a_6$. 

The functions $s,t,c$ are determined in terms of the remaining two variables $a_4,a_6$ by the following table of their boundary values.
\begin{center}\begin{tabular}{c | c | c | c | c|c|c|c|c}
$u$&$(-\infty,a_7)$&$(a_7,a_6)$&$(a_6,a_5)$&$(a_5,a_4)$&$(a_4,a_3)$&$(a_3,a_2)$&$(a_2,a_1)$&$(a_1,\infty)$\\\hline
$z$&$(0,1)$&$(1,\infty)$&$(\infty,-1)$&$(-1,0)$&$(0,1)$&$(1,\infty)$&$(\infty,-1)$&$(-1,0)$\\\hline
$s$&$1$&$1$&$0$&$0$&$1$&$1$&$0$&$0$\\
$t$&$0$&$1$&$1$&$0$&$0$&$1$&$1$&$0$\\
$c$&$0$&$-1$&$m_1-1$&$m_1-m_2$&$m_1-1$&$-1$&$0$&$0$
\end{tabular}\end{center}

For example 
{\footnotesize $$c=\frac1{\pi}\left(-\arg(\frac{u-a_6}{u-a_7}) +(m_1-1)\arg(\frac{u-a_5}{u-a_6})+(m_1-m_2)\arg(\frac{u-a_4}{u-a_5})+
(m_1-1)\arg(\frac{u-a_3}{u-a_4})-\arg(\frac{u-a_2}{u-a_3})\right).$$}
A short calculation shows 
that $c_u$ vanishes at the critical points of $z$ if
$$a_6=\frac{2-m_1-m_2}{m_2-m_1}.$$
Then the fact that $z=1$ at $a_7=\rho/a_2$ determines $a_4$:
$$a_4=-\frac{(\sqrt{2}+1)(2-m_1-m_2)}{2(1-m_2)}.$$

We thus have $(a_1,\dots,a_8) =$
{\footnotesize \be\label{as}=\left(0,-1,-\sqrt{2},-\frac{(1+\sqrt{2})(2-m_1-m_2)}{2(1-m_2)},-\frac{(4-2\sqrt{2})(1-m_2)}{m_1-m_2},
-\frac{2-m_1-m_2}{m_1-m_2},-\frac{\sqrt{2}(2-m_1-m_2)}{m_1-m_2},\infty\right).\ee}
This leads to the limit shape of Figure \ref{octagonlimitshape}.

\begin{figure}
\begin{center}\includegraphics[width=1.7in]{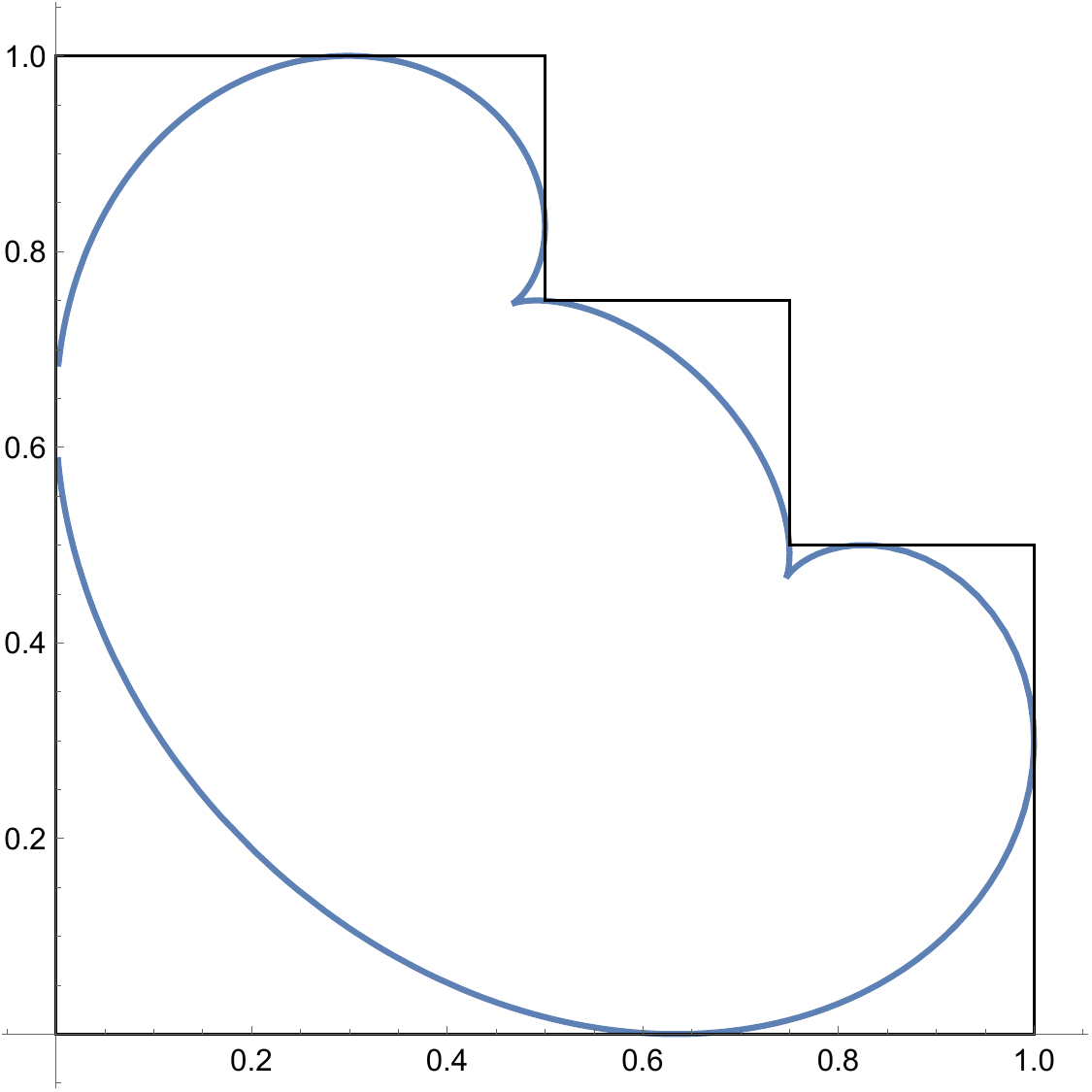}\hskip1cm\includegraphics[width=2in]{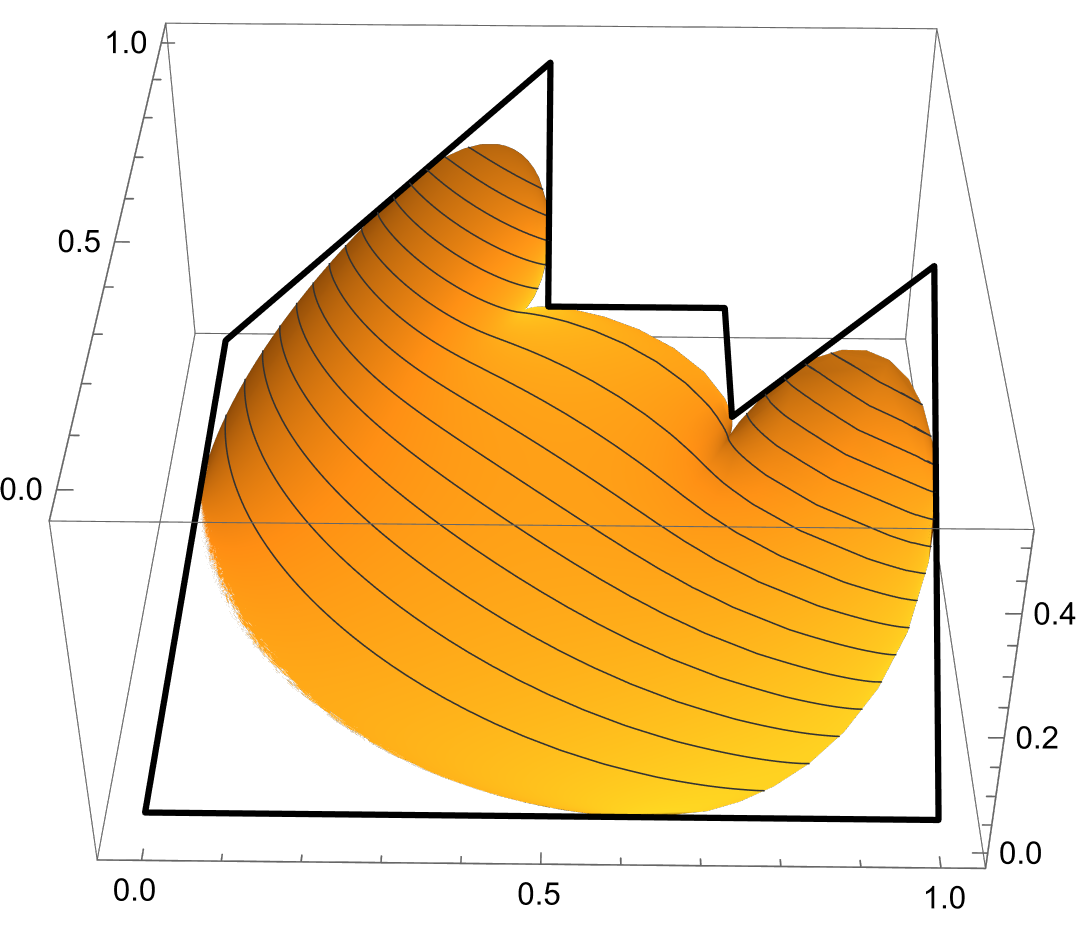}\end{center}
\caption{\label{octagonlimitshape}Domino arctic curve and rough phase region limit shape when $m_1=\frac12,m_2=\frac14$.}
\end{figure} 

There is a nontrivial inequality on $m_1,m_2$ in order for $R_0$ to be feasible (that is, in order for there to exist a tiling): 
this feasibility is equivalent to $(a_1,\dots,a_8)$ of (\ref{as}) being in the correct (cyclic) order on $\hat\R$. We need
$$m_1\ge m_2\ge (3+2\sqrt{2})m_1-2(1+\sqrt{2}).$$
See Figure \ref{octagonothers} for arctic curves near the two edges of the parameter space.
\begin{figure}
\begin{center}\includegraphics[width=2in]{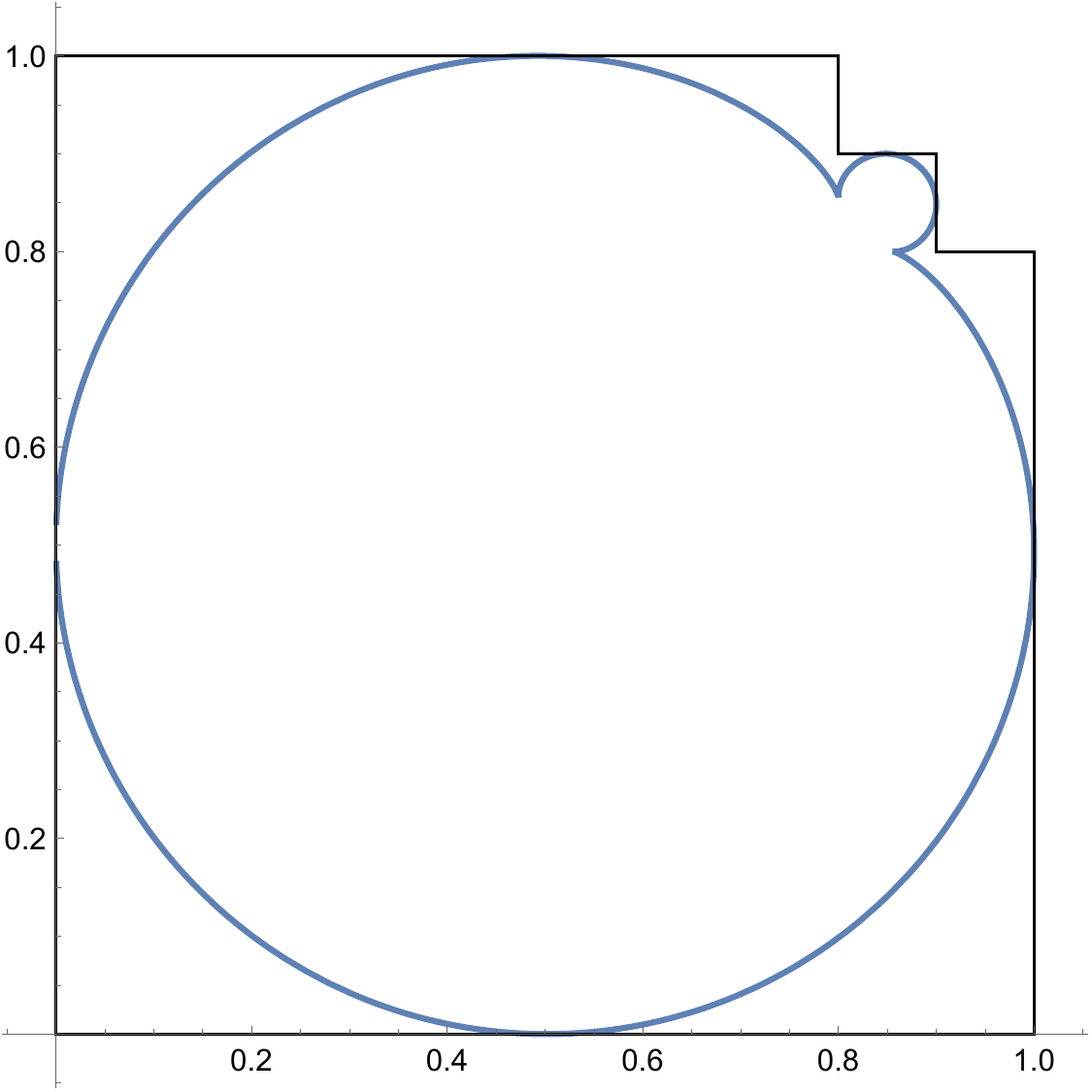}\hskip1cm\includegraphics[width=2in]{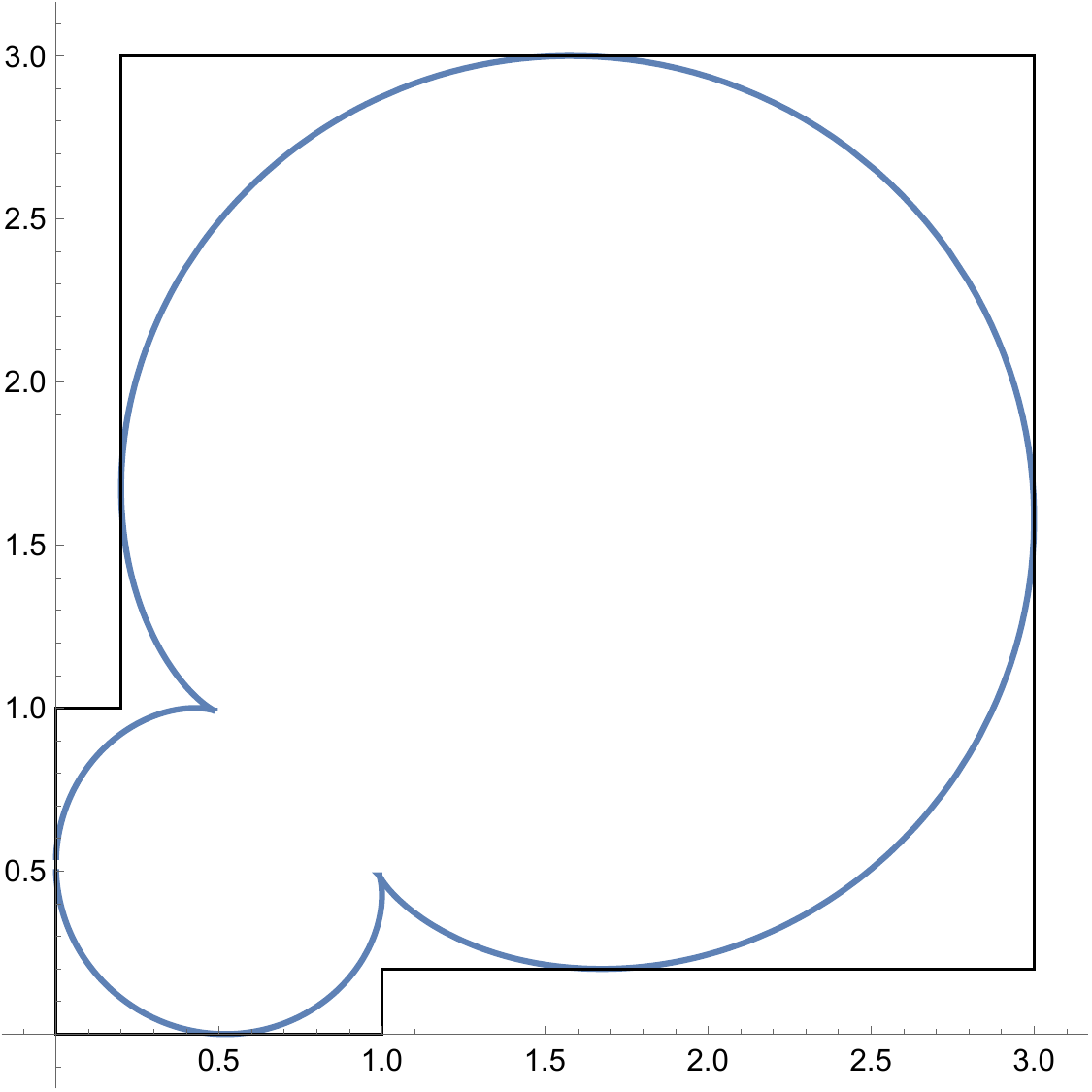}\end{center}
\caption{\label{octagonothers}Domino arctic curves when $m_1=\frac45,m_2=\frac1{10}$ and when $m_1=\frac15,m_2=-2$.}
\end{figure} 

\subsection{Aztec fortress}
\label{se:aztecfortress}

We consider the square-octagon lattice dimer model with weights $1,a$ on edges as shown in Figure \ref{period2domino}.  This dimer model is equivalent to the square lattice dimer model
with $2$-periodic weights, as shown in Figure \ref{period2domino}. 
\begin{figure}[h]
\begin{center}\includegraphics[width=2in]{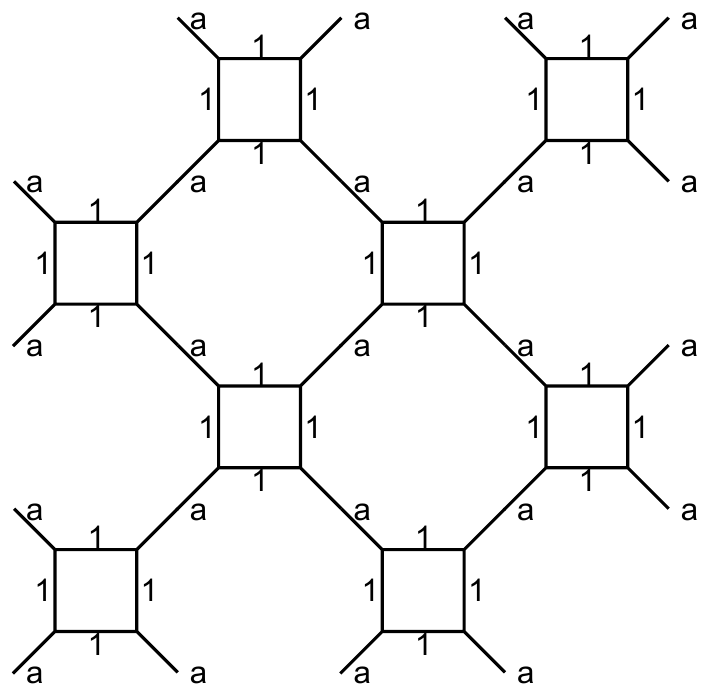}\hskip1cm\includegraphics[width=2in]{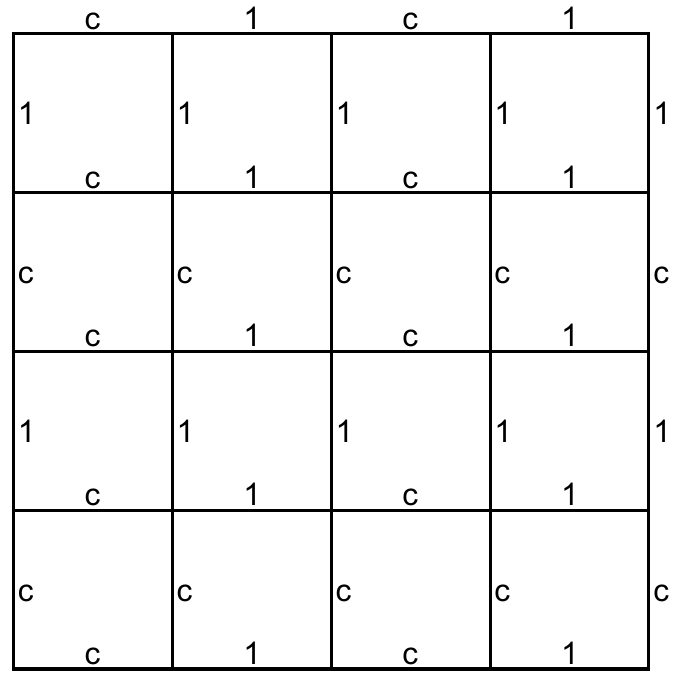}\end{center}
\caption{\label{period2domino} Square-octagon lattice and two-periodic square lattice. There is a weight-preserving correspondence between dimers on these two lattices if $2c=a^2$, see \cite{Propp}.}
\end{figure} 

The Newton polygon for either of these lattices is 
$$N=cvx\{(1,0),(0,1),(-1,0),(0,-1)\}.$$ The surface tension $\sigma$ is strictly convex in the interior of $N$, and analytic in $N$ except for a conical singularity at $(0,0)$ (when $c\ne1$). Here $z$, the conformal coordinate, is parameterized by an annulus which we choose to be $\A=[0,2]\times[0,\tau]/\!\!\sim$ where $\sim$ identifies the right and left boundaries by translation,
with $\tau=\tau(c)$.

The ``Aztec fortress" boundary conditions \cite{Propp} for the square octagon lattice are equivalent to the Aztec diamond boundary conditions for the square grid dimer model with the above $2$-periodic weights.
For these boundary conditions we can take $z=u$: the rough phase region maps to the annulus with degree $1$.  
The lower boundary of the annulus $\A$ corresponds to the outer arctic boundary, and the upper boundary of 
$\A$ corresponds to the inner (``smooth/rough") boundary.

Using the $4$-fold symmetry, the function $s$ has values $1,0,-1,0$ on the lower boundary of $\A$ for $z$ respectively in the intervals $[0,1/2],[1/2,1],[1,3/2],[3/2,2]$, and value $0$ in the upper boundary. The harmonic extension of these
boundary values allows $s$ to be written
explicitly in terms of elliptic functions.  It is convenient to use the Weierstrass $\sigma$-function (defined in e.g. \cite{Ahlfors}, not to be confused with the surface tension).
With $\tau=1$ we have
$$s(z)=1-\frac{\Im(z)}{2}-\frac1\pi\arg\frac{\sigma(z)\sigma(z+\frac12)}{\sigma(z+1)\sigma(z+\frac32)},$$
and likewise 
$$t(z)=\frac1\pi\arg\frac{\sigma(z-\frac12)\sigma(z-2)}{\sigma(z-1)\sigma(z-\frac32)},$$
$$c(z)=1-\frac{\Im(z)}{2}+\frac2\pi\arg\frac{\sigma(z)\sigma(z+\frac12)}{\sigma(z+1)\sigma(z+\frac32)}.$$
These expressions can be checked by differentiation, using $\frac{d}{dz}\log\sigma(z) = \zeta(z)$, 
where $\zeta$ is the Weierstrass zeta function, the analog of ``$1/z$" on the associated torus.
The limit shape for $\tau=1$ is shown in Figure \ref{ADfacet}.

\begin{figure}[h]
\begin{center}\includegraphics[width=2.5in]{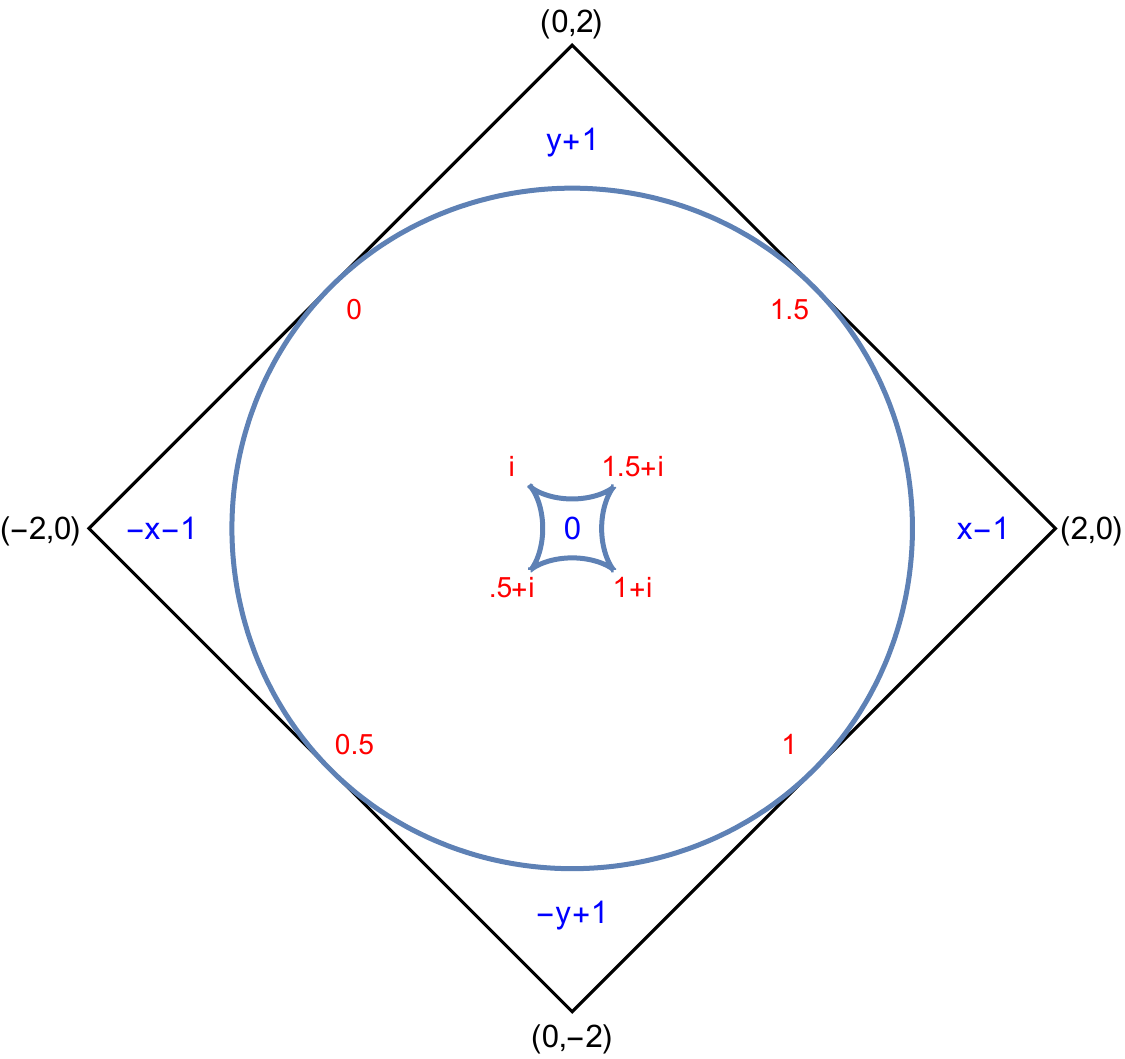}\hskip1cm\includegraphics[width=3in]{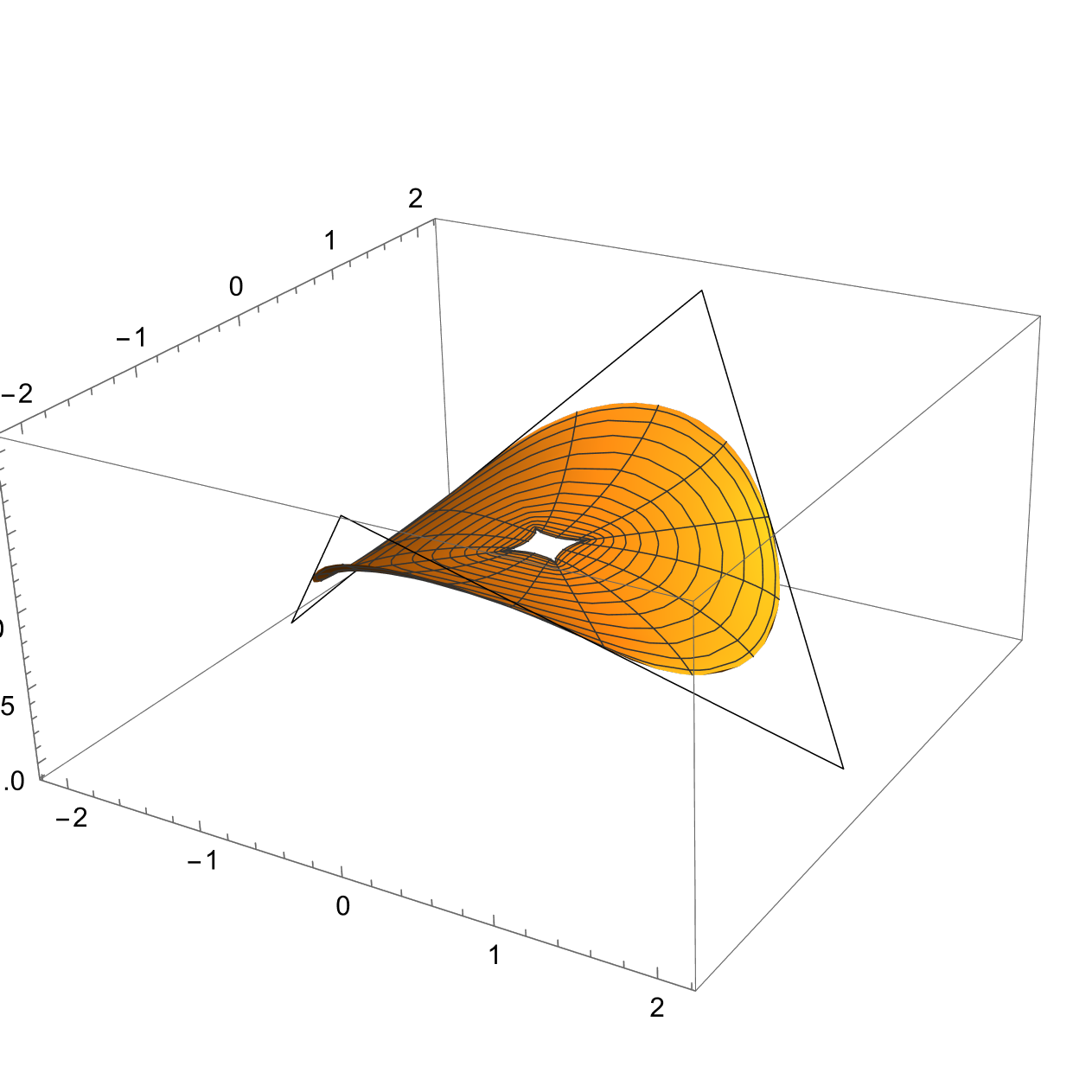}\end{center}
\caption{\label{ADfacet} For the square-octagon dimer model on a ``fortress" region, the arctic curve with facet equations (left) and 3d rough phase region limit shape (right).}
\end{figure}

\section{Five vertex model}

The five vertex model \cite{dGKW, KP2} is a model of non-intersecting north- and west-going lattice paths in
$\Z^2$. The probability of a configuration (on a finite domain) is proportional to the product of
its vertex weights as given in Figure \ref{5vvertexwts}.
\begin{figure}[htbp]
\begin{center}\includegraphics[width=3.5in]{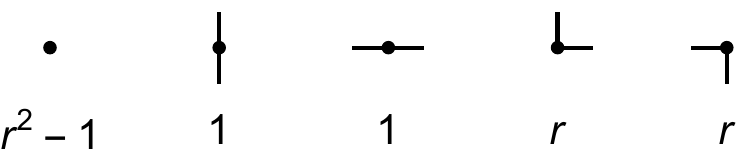}\end{center}
\caption{\label{5vvertexwts}Vertex weights for the $5$-vertex model.}
\end{figure} 
Here $1<r<\infty$ is a parameter\footnote{One can also define and solve the model for $0<r<1$, setting the weight of the empty vertex to be $1-r^2$, see \cite{dGKW, KP2}. We will not consider this case here.}.
 
The features of the $5$-vertex model which are relevant for the current discussion are as follows.
\begin{enumerate}
\item Let $N=cvx\{(0,0),(1,0),(0,1)\}$; the surface tension $\sigma=\sigma(s,t)$ is strictly convex and analytic on
the interior of $N$ and $+\infty$ outside of $N$.
\item The intrinsic coordinate $z$ can be chosen to be parameterized by $z\in\bar\H$. There is another intrinsic coordinate
$w$, where $w\in\H$, related to $z$ through the spectral curve $P(z,w)=1-z-w+(1-r^2)zw=0.$

\item For $(s,t)\in N$ we have $t(z)=\frac{\pi+\Arg z}{\theta}$ and $s(z)=t(\bar w)= \frac{\pi-\Arg w}{\theta}$
where $\theta(z)=2\pi+\Arg\frac{z}{1-z}=2\pi-\Arg\frac{w}{1-w}$. See Figure \ref{5vst}.
\end{enumerate} 

\begin{figure}
\begin{center}\includegraphics[width=5in]{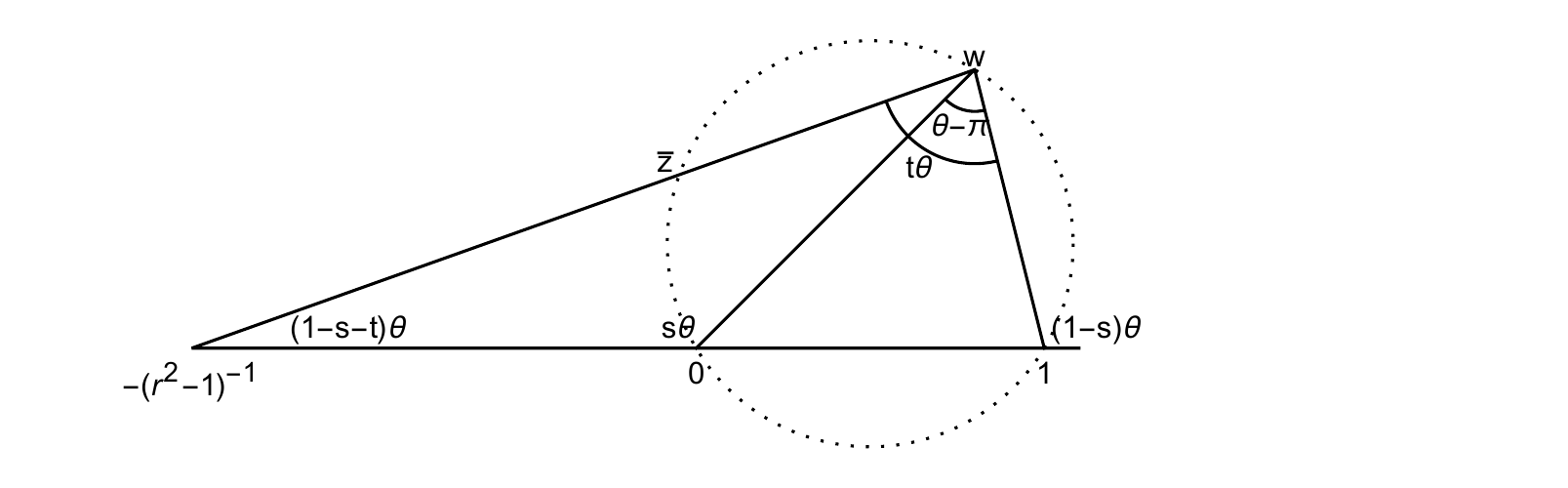}\end{center}
\caption{\label{5vst}Relationship between $r,s,t,w,z,\theta$ in the $5$-vertex model. Given $r,w$
we can draw a circle passing through $0,1,w$. Then $z,\theta,s,t$ are determined as shown.}
\end{figure} 
The five vertex model limit shape is the envelope of the family of planes parameterized by $u\in\H$ (we use $x_3$ for the third coordinate)
\be\label{5vlines1}
x_3=sx+ty + c.
\ee
The surface itself is the envelope of these planes; it is obtained by solving simultaneously both (\ref{5vlines1}) and its derivative which is (see \cite[Corollary 4.3]{KP2})
\be\label{5vlines}
x_3=\frac{(\theta s)_u}{\theta_u} x+\frac{(\theta t)_u}{\theta_u} y + \frac{(\theta c)_u}{\theta_u}\ee
where $z=z(u)$ is an appropriate holomorphic map. 
Since $\theta, \theta s,\theta t$ and $\theta c$ are harmonic, they can be determined
from their boundary values, that is, for $u\in\R$: see a worked example in the next section.

\subsection{5-vertex boxed plane partition}

Here we determine the limit shape for the 5-vertex model in a hexagon. See Figure \ref{hexagon}. This limit shape was
first determined in \cite{dGKW} by a different method; the method presented here leads to a more explicit form of the limit shape and is more amenable to generalizations.
We parametrize the rough phase region by $u \in \bar\H$ (via an orientation-preserving diffeomorphism in this case), and $z(u)\in\bar\H$ is a degree-2 cover, to be determined. 
There are $8$ tangency points to $R_0$ of the arctic curve, where $u$ takes real values $a_1, \dots, a_8$.
We set $a_8=\infty, a_1=0$, and order them as $a_7\le a_6\le\dots\le a_1=0$.  
\begin{figure}
\begin{center}\includegraphics[width=2.5in]{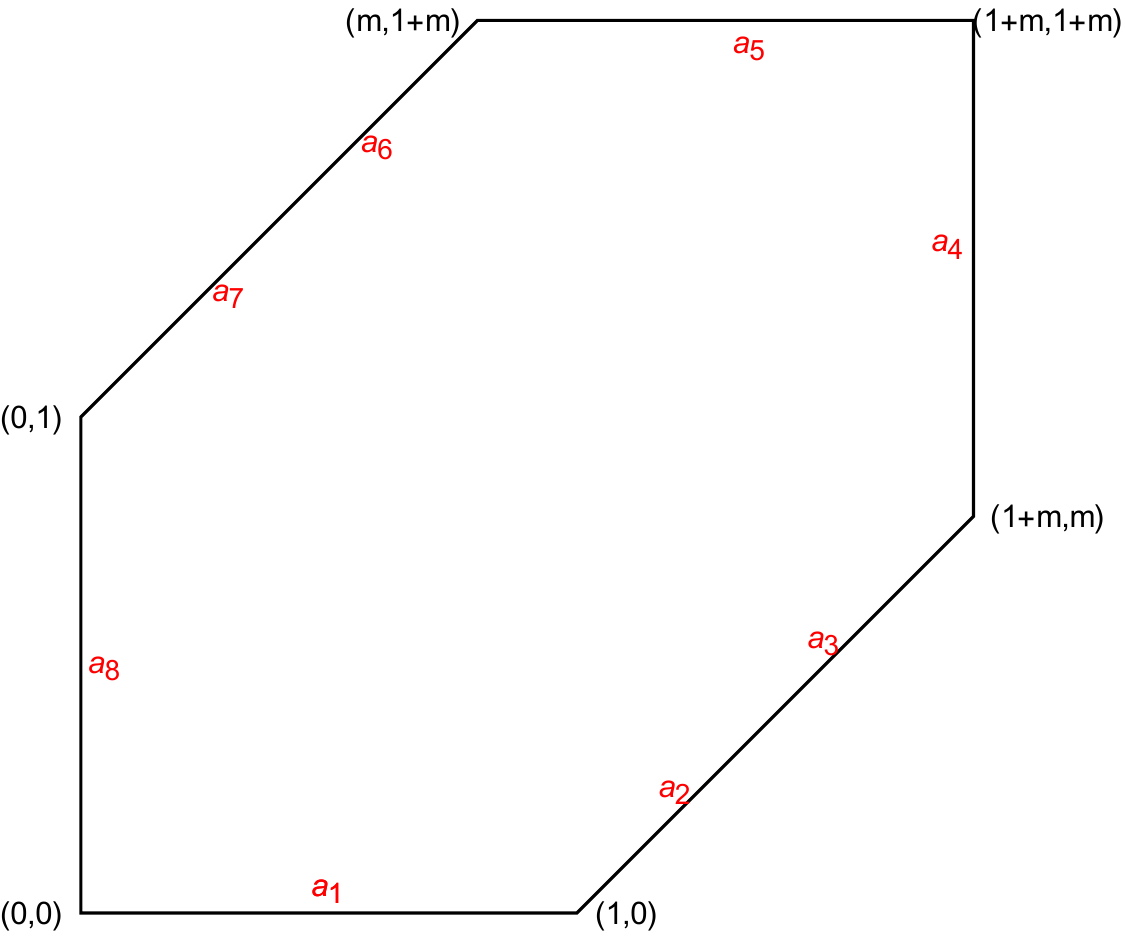}\end{center}
\caption{\label{hexagon}}
\end{figure} 
We determine the intercept function $c=c(u)$ as follows.
Recall that $c(u)$ is a ratio of two harmonic functions, $c(u)=\frac{G(u)}{\theta(u)}$.
The boundary values of $c$ and $\theta$ are determined from the following table, which allows
us to find the boundary values of $G$. The harmonicity of $G$ allows us to define $G(u)$ in all of $\H$,
from which we get $c(u)$. 

\begin{center}\begin{tabular}{c | c | c | c | c|c|c|c|c}
$u$&$(-\infty,a_7)$&$(a_7,a_6)$&$(a_6,a_5)$&$(a_5,a_4)$&$(a_4,a_3)$&$(a_3,a_2)$&$(a_2,a_1)$&$(a_1,\infty)$\\\hline
$z$&$(-1,0)$&$(0,1)$&$(1,\infty)$&$(-\infty,-1)$&$(-1,0)$&$(0,1)$&$(1,\infty)$&$(-\infty,-1)$\\\hline
$s$&$1$&$1/2$&$0$&$0$&$1$&$1/2$&$0$&$0$\\
$t$&$0$&$1/2$&$1$&$0$&$0$&$1/2$&$1$&$0$\\
$c$&$0$&$-1/2$&$-1$&$m$&$-1$&$-1/2$&$0$&$0$\\
$\theta$&$\pi$&$2\pi$&$\pi$&$\pi$&$\pi$&$2\pi$&$\pi$&$\pi$\\
$\frac{G}{\pi}$&$0$&$-1$&$-1$&$m$&$-1$&$-1$&$0$&$0$
\end{tabular}\end{center}

This gives 
$$G = -\arg(\frac{u-a_7}{u-a_6})-\arg(\frac{u-a_6}{u-a_5})+m\arg(\frac{u-a_5}{u-a_4})-\arg(\frac{u-a_4}{u-a_3})-\arg(\frac{u-a_3}{u-a_2}).$$

Using symmetry along the major diagonal we set $a_5=1/a_4,a_6=1/a_3,a_7=1/a_2$. Symmetry along the minor diagonal
then relates the $a$s under the map $u\mapsto\frac{u-a_4}{a_4u-1}$. This gives $a_3=\frac{a_2-a_4}{a_2a_4-1}$,
and there are two remaining parameters $a_2,a_4$. 

We have $z=B\frac{(u-a_3)(u-a_7)}{u(u-a_5)}$ and $B=-1$ by setting $u=\infty$.
Also $1=z(a_2)$ sets one of the remaining parameters. The last parameter is determined by $G'(u)=0$ when $z'(u)=0$. 
This leads to 
$$(a_1,\dots,a_4) = \left(0,-\frac{\sqrt{m}}{2^{1/4}\sqrt{m+2}},-\frac{2^{1/4}\sqrt{m}}{\sqrt{m+2}},-\frac{2+\sqrt{2}}{2^{7/4}}\frac{\sqrt{m(m+2)}}{m+1}\right)$$
and $(a_8,a_7,a_6,a_5)$ are inverses of these. This leads to the limit shape of 
Figure \ref{bpp53d}.

\begin{figure}
\begin{center}\includegraphics[width=2.5in]{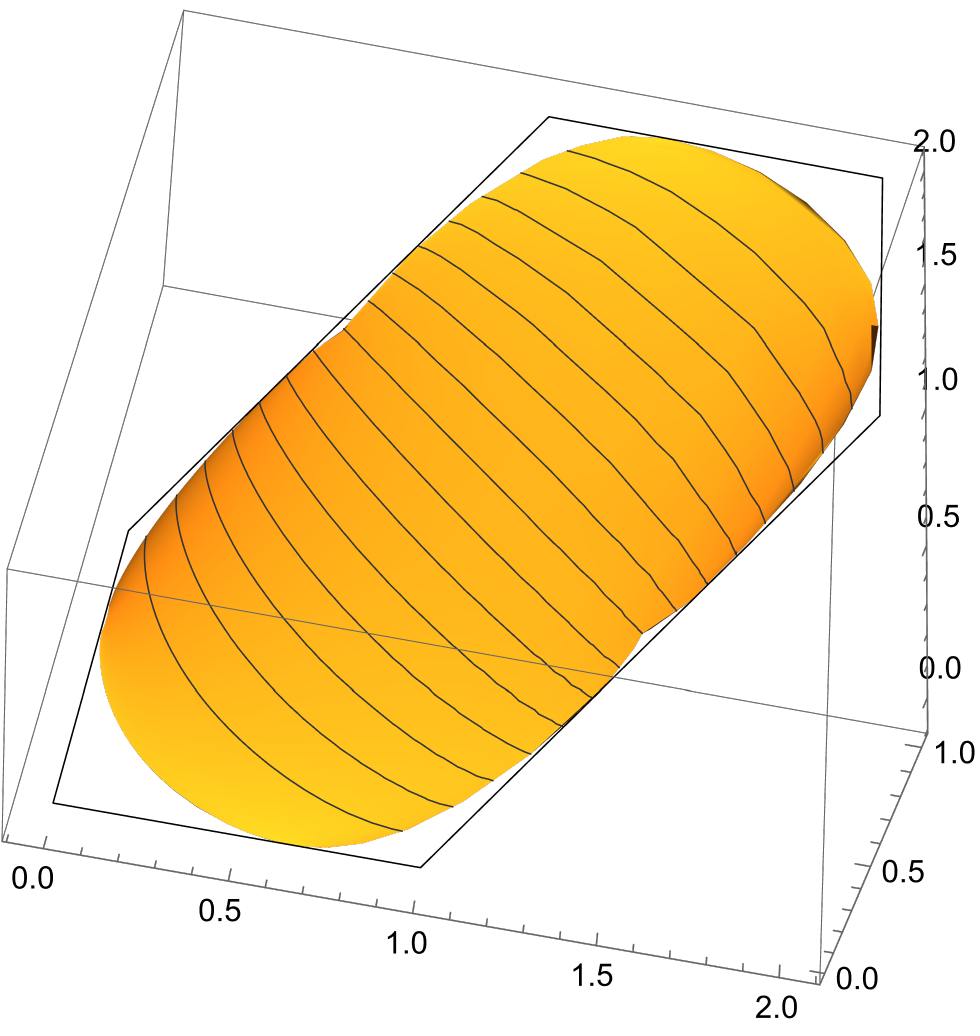}\end{center}
\caption{\label{bpp53d}The 5-vertex model with ``boxed plane partition" boundary conditions, showing height contours,
with $r=\sqrt{2}, m=1$. Observe the small facets in the middle of the diagonal edges (pictured in white).}
\end{figure}

\section{Four-vertex model versus lozenge tilings}\label{4vtx}

In this section, we consider the four-vertex model which is a degenerate case of the five-vertex model, where the 
third type of vertex in Figure \ref{5vvertexwts} has weight $0$.
Its arctic curves for specific (analogous to boxed plane partition) boundary conditions were recently studied in \cite{burenev2023arctic}. In order to be consistent with \cite{burenev2023arctic} we also switch to north- and east-going lattice paths in this discussion, i.e. mirror the last two configurations with respect to the vertical axis in Figure \ref{5vvertexwts}.  We find that limit shapes of the four-vertex model, viewed as surfaces in three dimensions, are, after a certain fixed linear transformation, lozenge tiling limit shapes. The linear transformation in our setup is a shearing transformation of the form
\begin{equation}
\label{eq:shearh}
 (x,y,h) \mapsto (x,y-x-h,h.) 
\end{equation}
This follows from a simple bijection on the level of configurations: see Figure \ref{5to4}.
\begin{figure}
\begin{center}\includegraphics[width=1.7in]{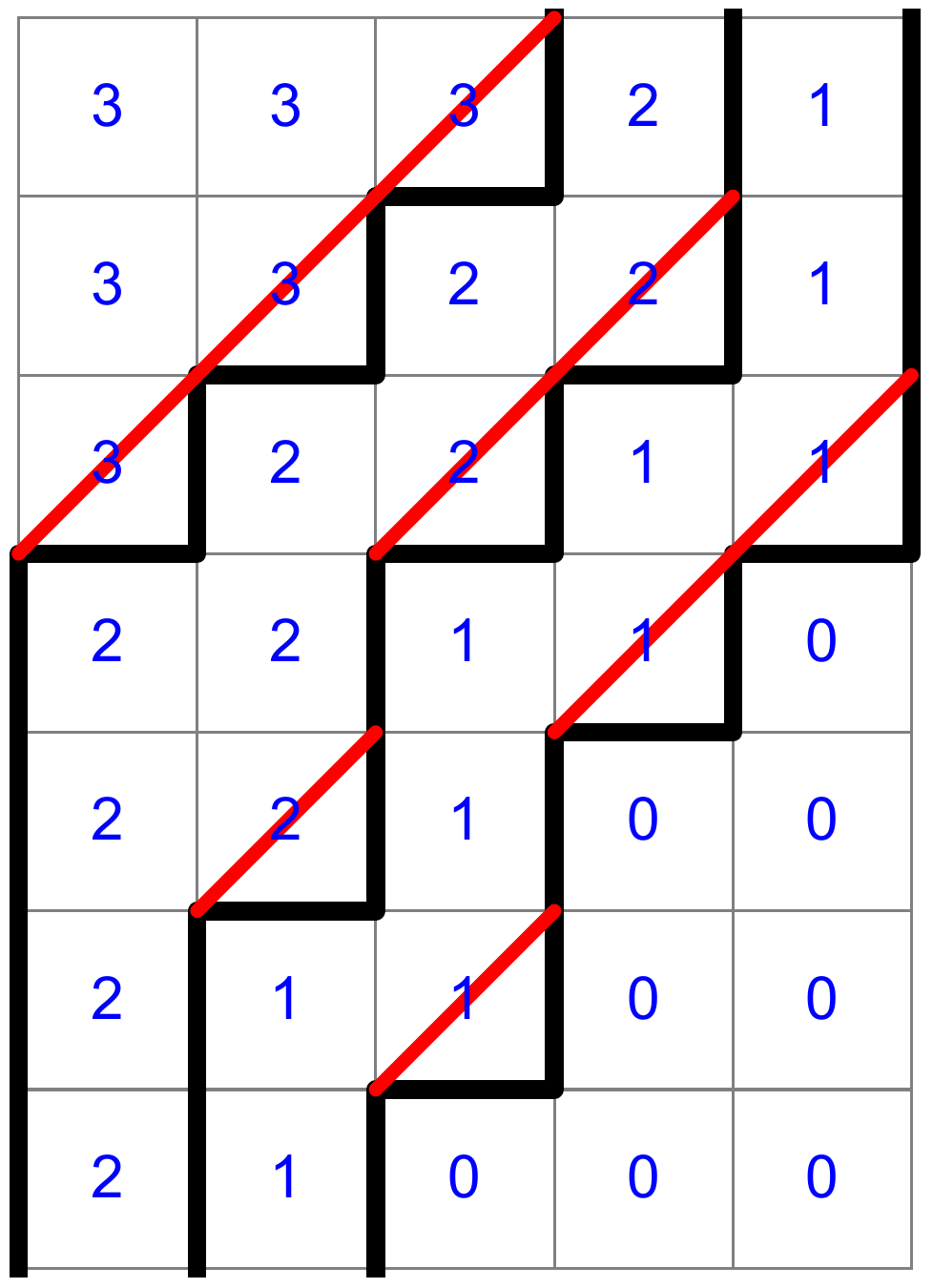}\hskip1cm\includegraphics[width=1.7in]{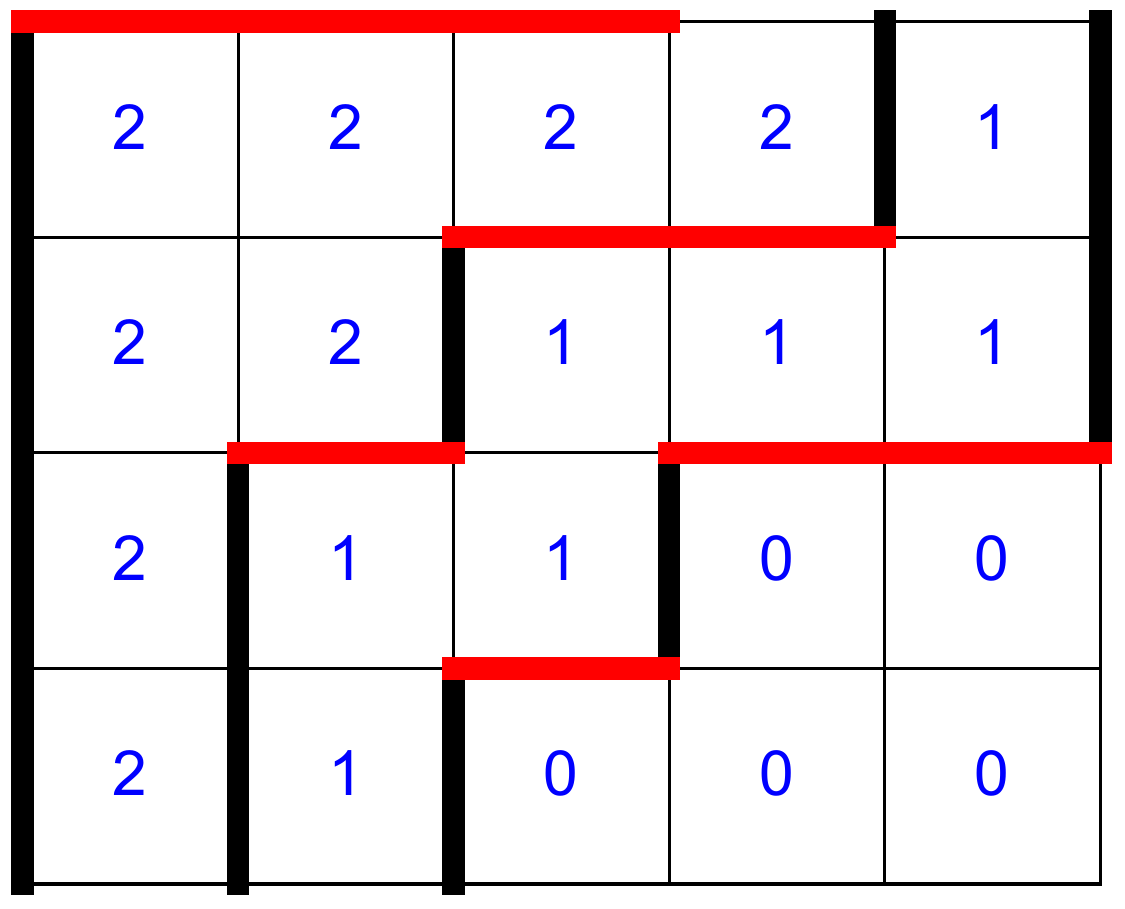}\end{center}
\caption{\label{5to4} Left: a 4-vertex configuration in black, with height function on faces. Each horizontal black edge has a vertical edge immediately to its right; join these with a diagonal edge as shown (red). Then applying a linear map $(x,y,h)\to(x,y-x-h,h)$, the images
of the vertical black edges and diagonal red edges form a standard 
monotone non-intersecting NE lattice path configuration (right), or equivalently, a lozenge tiling configuration. Each red segment replaces two corners in $4$-vertex model; therefore giving a red segment a weight $r^2$, this becomes
a weight-preserving bijection from $4V$ to lozenges with weight $1$ for vertical edges and $r^2$ for horizontal edges
(and $r^2-1$ for empty vertices). }
\end{figure} 

From the five-vertex degeneration point of view, one would expect the surface tension of the four-vertex model to satisfy the trivial potential property. This is indeed the case which we explain below. We find it interesting that a simple linear transform leads to a trivial potential property.

Our conventions for the height functions, see Figure \ref{5to4}, induces the Newton polygons
\[ \tilde N=cvx \{ (0,0),(-1,0),(-1/2,1/2) \}  \quad \text{and} \quad N=cvx \{ (0,0),(-1,0),(0,1) \} \]
for the four-vertex model (left side) and a lozenge tiling model (right side), respectively.

Let $\sigma$ be the surface tension of the lozenge model. This has a representation in terms of the intrinsic $z \in \H$ with harmonic functions $s(z),t(z)$ and $X(z),Y(z)$ such that 
\begin{equation}
\label{eq:harmonicconjugates}
X_z+i \pi t_z=0 \quad Y_z-i \pi s_z=0,
\end{equation}
and the correspondence $\nabla \sigma (s,t)=(X,Y)$, see \cite{KP1}.
The Legendre dual of $\sigma$, the free energy, is denoted by $F(X,Y)$. Its graph $(X,Y,F(X,Y))$ describes the corresponding Wulff shape in $\R^3$.
Because of the above correspondence \eqref{eq:shearh}, the Wulff shape of the 4-vertex model is also a three dimensional shearing of the lozenge tiling Wulff shape, by considering the inverse:
\begin{equation}
\label{eq:shearing}
(X,Y,F) \mapsto (\tilde X,\tilde Y, \tilde F)=(X,Y+X+F,F)
\end{equation}

We record next the effect of this linear change on the surface tension.
The relevant transformation between the Newton polygons $N \to \tilde N$ is
\[ (s,t) \mapsto (\tilde s,\tilde t)=(\frac{s-t}{1+t},\frac{t}{1+t}) \]
The surface tension $\tilde\sigma$ of the 4-vertex model is
\[ \tilde \sigma(\tilde s, \tilde t)=(1-\tilde t) \, \sigma (\frac{\tilde s+ \tilde t}{1-\tilde t},\frac{\tilde t}{1- \tilde t})=\frac{\sigma(s,t)}{1+t}.\]
Indeed, by a direct calculation one verifies that
\[ \tilde \sigma_{\tilde s}(\tilde s, \tilde t)=\sigma_s(s,t)=X \quad \text{and} \quad \tilde \sigma_{\tilde t}(\tilde s, \tilde t)=Y+X+F(X,Y). \]

The same variable $z \in \H$ parametrises both the slopes $(\tilde s,\tilde t)$ and the 
external fields $(\tilde X, \tilde Y)$, by the indicated change of variables. We now verify that $z$ is an 
intrinsic coordinate and the model has a trivial potential. We find that
\[ \tilde X_z +  i \pi (1+t)^2 \tilde t_z=X_z+ i \pi t_z=0,\]
and (using $F_z=F_XX_z+F_YY_z=sX_z+tY_z$)
\[ \tilde Y_z- i \pi (1+t)^2  \tilde s_z=(1+t) (Y_z - i \pi s_z) +(1+s)(X_z+ i \pi t_z)=0,\]
because of \eqref{eq:harmonicconjugates}. This means by Proposition 2.1 of \cite{KP1} that $z$ is an intrinsic coordinate for $\tilde \sigma$ and $\kappa(z)=\sqrt{\det H_{\tilde \sigma}}=\pi(1+t)^2$. Furthermore, since $\kappa^{1/2}(z)=\sqrt{\pi}(1+t(z))$ is a harmonic function, $\tilde \sigma$ has a trivial potential. 

In fact, we observe that the surface tension $\tilde \sigma$ has the same structure as the general large $r$ 5-vertex model (see Theorem 3.2 b of \cite{KP2}): the boundary values of $\sqrt{\det H_{\tilde \sigma}}$ are equal to $\pi$ in the corners $(0,0)$ and $(-1,0)$ and $4 \pi$ in the corner $(-1/2,1/2)$, and the extension inside is dictated by the harmonicity of $\kappa^{1/2}(z)$. 
The behaviours at these two types of corners are referred to as ferroelectric and anti-ferroelectric phases in \cite{burenev2023arctic}.

\begin{rmk}
In the example of \cite{burenev2023arctic}, for the so-called `scalar-product' boundary conditions, the arctic curve is found to be consisting of six arcs, each of them is a portion of certain ellipses. This can be explained via the reduction to lozenge model, as follows. It is easily seen that the transformation \eqref{eq:shearh} deforms the scalar-product boundary conditions to hexagonal (or ``boxed plane-partition'') boundary conditions for the lozenge model. Its well-known arctic curve is the inscribed ellipse to the hexagon. Note that the height function $h$ is affine on each six arcs of the ellipse between tangency points. Thus when we move back to the 4-vertex setting with the inverse of \eqref{eq:shearh}
\[ (x,y) \mapsto (x,y+x+h),
\]   
the single ellipse is transformed to six different ellipses depending on the behaviour of $h$.
Note, however, that the arctic curves viewed as curves in three dimensions are related by a single linear transformation. Only
when projected to the $(x,y)$-plane do the six different linear transformations become apparent.
\end{rmk}

\old{
Next we show that minimisers for the surface tension $\tilde \sigma$ arise under the same shear transformation as in \eqref{eq:shearing} from minimisers of $\sigma$. Let us denote $\psi=1-t$. 
The height function $\tilde h(\tilde x,\tilde y)$ for the surface tension $\tilde \sigma$ can be found by solving a system
\begin{align*}
\tilde h &= \tilde s \tilde x +\tilde t \tilde y + \tilde c\\
0 &= \tilde s_u \tilde x +\tilde t_u \tilde y + \tilde c_u.\\
\end{align*}
Here $s(u)=s(z(u))$, $t(u)=t(z(u))$, with $z(u)$ holomorphic and $\tilde c$ is a ratio of a harmonic function $c$ and $\psi$.
The second equation, similarly as in \eqref{5vlines}, can be rewritten as
\[ s_u \tilde x + t_u \tilde y + c_u - \psi_u \tilde h =0.
\]
Since $\psi=1-t$, $\psi_u=-t_u$ and we can thus rewrite this as
\[ s_u \tilde x + t_u (\tilde y+ \tilde h) + c_u =0.
\]
This means that with the choice
\[ x=\tilde x, \quad y=\tilde y+\tilde h,
\]
we can define a new height function $h=sx+ty+c$ which solves a system corresponding to minimisation for $\sigma$. Observe that
\[\psi \tilde h = \psi (\tilde s \tilde x +\tilde t \tilde y + \tilde c)=s \tilde x+ t \tilde y +c=sx+ty+c- t \tilde h=h+(\psi-1) \tilde h.
\]
We find that $\tilde h=h$.
This means that the shear transformation
\[ (x,y,h) \mapsto (\tilde x,\tilde y,\tilde h)=(x,y-h,h)
\]
maps a minimiser of $\sigma$ to a minimiser of $\tilde \sigma$.
}

\bigskip
\noindent\textbf{Acknowledgements.} 
IP is grateful to the Workshop on ‘Randomness, Integrability, and Universality’, held in Spring 2022 at the Galileo Galilei Institute for Theoretical Physics, as well as to Yale University for hospitality during his visit in December 2022. RK thanks the University of Helsinki for a visit in summer 2023 where some of this
research was performed. We acknowledge Beno\^{\i}t Laslier for a simultaneous observation about the linear relationship between 4-vertex and lozenge limit shapes. We thank the anonymous referees for their valuable comments.

\bibliographystyle{hplain}
\bibliography{trivpotential3}

\end{document}